\documentclass{amsart}
\usepackage[numbers]{natbib}

\usepackage{amsmath,amstext, amsthm, amssymb,color,enumitem,pslatex}

\usepackage[english]{babel}
\usepackage[all]{xy}

\definecolor{e-mail}{rgb}{0,.40,.80}
\definecolor{reference}{rgb}{.20,.60,.22}
\definecolor{citation}{rgb}{0,.40,.80}

\usepackage[colorlinks, linkcolor=reference,
            citecolor=citation,
           urlcolor=e-mail]{hyperref}

\newtheorem{theorem}{Theorem}[section]

\newtheorem{lemma}[theorem]{Lemma}
\newtheorem{proposition}[theorem]{Proposition}

\theoremstyle{remark}
\theoremstyle{definition}

\newtheorem{remark}[theorem]{Remark}
\newtheorem{definition}[theorem]{Definition}

\newtheorem{example}[theorem]{Example}

\numberwithin{equation}{section}

\def\beq{\begin{equation}}
\def\eeq{\end{equation}}

\definecolor{todo}{rgb}{1,0,0}

\definecolor{bl}{rgb}{0,0,1}
\newcommand{\bl}[1]{\textcolor{bl}{#1}}

\def\Q{{\mathbb Q}}

\def\I{{\mathbb I}}

\def\K{{\mathbf k}}
\def\k{{\mathbf k}}

\def\o{\omega}

\def\p{\prime}

\def\cM{{\mathcal M}}
\def\cN{{\mathcal N}}

\def\cU{{\mathcal U}}
\def\cV{{\mathcal V}}
\def\cW{{\mathcal W}}

\def\ra{\mathfrak{r}}

\def\p{\partial}

\def\Hom{\operatorname{Hom}}
\def\Gl{\operatorname{GL}}
\def\GL{\operatorname{GL}}
\def\End{\operatorname{End}}

\def\Sl{\operatorname{SL}}
\def\SL{\operatorname{SL}}
\def\Mat{\operatorname{Mat}}

\def\Gal{\operatorname{Gal}}
\def\Galdelta{\operatorname{Gal}^\delta}

\def\Ru{\operatorname{R_u}}
\def\Ker{\operatorname{Ker}}
\def\Diff{\operatorname{Diff}}
\def\Lie{\operatorname{Lie}}
\def\trdeg{\operatorname{tr.deg}}

\def\Quot{\operatorname{Quot}}

\def\diag{\operatorname{diag}}

\def\Ga{\bold{G}_a}

\def\Vect{\operatorname{Vect}}

\renewcommand{\ker}{\mathrm{Ker}}

\begin{document}
\title[Third order linear differential equations with parameters]{Calculating   Galois groups of third order linear differential equations with parameters} 
\author[Andrei Minchenko]{Andrei Minchenko}
\author[Alexey Ovchinnikov]{Alexey Ovchinnikov}
\address{University of Vienna, Department of Mathematics, Oskar-Morgenstern-Platz 1, Vienna, Austria}
\email{an.minchenko@gmail.com}
\address{Department of Mathematics\\
CUNY, Queens College \\
65-30 Kissena Blvd \\
Queens, NY 11367, USA }
\email{aovchinnikov@qc.cuny.edu}

\subjclass[2010]{12H05, 12H20, 13N10, 20G05, 20H20}
\keywords{Differential equations with parameters, differential Galois groups, algorithms}
\thanks{This work was partially supported  
by the Austrian Science Foundation FWF, grant P28079, 
by the NSF grants CCF-0952591 and DMS-1606334, the NSA grant \#H98230-15-1-0245,  CUNY CIRG \#2248, and  PSC-CUNY grant \#69827-00 47.}

\begin{abstract} Motivated by developing algorithms that decide hypertranscendence of solutions of extensions of the Bessel differential equation, algorithms computing the unipotent radical of a parameterized differential Galois group have been recently developed. Extensions of Bessel's equation, such as the Lommel equation, can be viewed as homogeneous parameterized linear differential equations of the third order. In the present paper, we give the first known algorithm that calculates the differential Galois group of a third order parameterized linear differential equation.
\end{abstract}
\maketitle

\section{Introduction}
The existing methods of studying hypertranscedence of solutions of second order inhomogeneous linear differential equations \cite{HMO} use results from  parameterized differential Galois theory \cite{Dreyfus:density,MitSingMonod}. In this application, these results essentially limit the consideration of inhomogeneous terms to rational functions. However, the Lommel function with a non-integer parameter $\mu$, the Anger and Weber functions are all examples of solutions of inhomogeneous versions of Bessel's differential equation with inhomogeneous terms that are not rational functions. Therefore, new techniques in parameterized differential Galois theory are needed so that the theory can be readily applied to the above functions. In addition, third order linear differential equations with parameters appear in  electromagnetic waves modeling \cite[\S 78]{LL}.

The hypertranscendence results of~\cite{HMO} are based in part on new results in the representation theory of linear differential algebraic groups. In particular, extensions of a trivial representation by a two-dimensional representation are analyzed there. In the present paper, we give a deeper development of the representation theory by analyzing three-dimensional representations of linear differential algebraic groups that are not necessarily extensions of a trivial representation. Based on this, we present the first algorithm that calculates the differential Galois group of a third-order homogeneous parameterized  linear differential equation whose coefficients are rational functions.

Our algorithm relies on several existing algorithms in the parameterized differential Galois theory referenced below.
For linear differential equations of order two,   an algorithmic development was initiated in~\cite{Dreyfus}  and  completed in~\cite{ArrecheJSC}. An algorithm that allows to test 
if the parameterized differential Galois group is reductive and to compute the group in that case can be found in~\cite{MinOvSing}. In \cite{MinOvSingunip}, it is shown how to compute
the parameterized differential Galois group if its quotient by the unipotent radical is conjugate to a group of matrices with constant entries with respect to the parametric derivations.  The algorithms of \cite{MinOvSingunip,MinOvSing} rely on the algorithm of computing differential Galois groups~\cite{Hrushcomp}, which has been further analyzed and improved in~\cite{RFdifferential2015}, in the case of no parameters.

The paper is organized as follows. Section~\ref{sec:BasicDef} contains basic definitions and facts in differential algebraic geometry, differential algebraic groups, and systems of linear differential equations with parameters phrased in the language most suitable for our description of our algorithm. The main algorithm is described in Section~\ref{sec:Algorithm}.
\section{Basic definitions and facts}\label{sec:BasicDef}
For the convenience of the reader, we present the basic standard definitions and facts from differential algebraic geometry, differential algebraic groups, and differential modules used in the rest of the paper.
\subsection{Differential algebraic geometry}
\begin{definition} A {\em differential ring} is a ring $R$ with a finite set $\Delta=\{\delta_1,\ldots,\delta_m\}$ of commuting derivations on $R$. A {\em $\Delta$-ideal} of $R$
is an ideal of $R$ stable under any derivation in $\Delta$. \end{definition}

 In the present paper, $\Delta$ will consist of one or two elements. Let $R$ be a $\Delta$-ring. For any $\delta \in \Delta$, we denote
 $$
 R^\delta= \{r \in R\:|\: \delta(r) = 0\},
 $$
 which is a $\Delta$-subring  of $R$ and is called the {\em ring of $\delta$-constants} of $R$. If $R$ is a field and a differential ring, then it is called a differential field, or $\Delta$-field for short. For example, $(R = \Q(x,t), \Delta=\{\partial/\partial x,\partial/\partial t\})$ is a differential field. 
 
The ring of $\Delta$-{\em differential polynomials} $K\{y_1,\ldots,y_n\}$ in the differential indeterminates, or $\Delta$-indeterminates,  $y_1,\ldots,y_n$ and with coefficients in a $\Delta$-field $(K,\Delta)$,  is the ring of polynomials in the indeterminates formally denoted $$\left\{\delta_1^{i_1}\cdot\ldots\cdot\delta_m^{i_m} y_i\:\big|\: i_1,\ldots,i_m\ge 0,\, 1\le i\le n\right\}$$ with coefficients in $K$. We endow this ring with a structure of a $\Delta$-ring by setting 
$$\delta_k \left(\delta_1^{i_1}\cdot\ldots\cdot \delta_m^{i_m} y_i \right)= \delta_1^{i_1} \cdot\ldots\cdot \delta_k^{i_k+1} \cdot\ldots\cdot \delta_m^{i_m} y_i.$$ 

\begin{definition}[{see \cite[Corollary~1.2(ii)]{Marker2000}}]
A differential field $(K,\Delta)$ is said to be differentially closed or $\Delta$-closed for short, if, for every finite set of $\Delta$-polynomials $F \subset K\{y_1,\ldots,y_n \}$, if the system of differential equations $F=0$ has a solution  with entries in some $\Delta$-field extension $L$, then it has a solution with entries in $K$.\end{definition}

Let  $(\K,\delta)$ be a differentially closed field, $C=\K^\delta$, and  $(F,\delta)$  a $\delta$-subfield of $\k$.

\begin{definition}\label{def:Kc} A {\it Kolchin-closed} (or $\delta$-closed, for short) set $W \subset \K^n$   is the set of common zeroes
of a system of $\delta$-polynomials with coefficients in $\K$, that is, there exists  $S \subset \K\{y_1,\dots,y_n\}$ such that
$$
W = \left\{ a \in \K^n\:|\: f(a) = 0 \mbox{ for all } f \in S \right\}.$$
We say that $W$ is defined over $F$ if $W$ is the set of zeroes of  $\delta$-polynomials with coefficients in  $F$. More generally, for any $\delta$-ring $R$ extending $F$, $$
W(R) = \left\{ a \in R^n\:|\:   f(a) = 0 \mbox{ for all } f \in S \right\}.$$
\end{definition}

\begin{definition}
If $W \subset \K^n$ is a  Kolchin-closed set defined  over $F$,  the $\delta$-ideal  $$\I(W) = \{ f\in F\{y_1,  \ldots , y_n\} \ | \ f(w) = 0 \mbox{ for all } \ w\in W(\K)\}$$
is called the defining $\delta$-ideal of $W$ over $F$.
Conversely, for a  subset  $S$ of $F\{y_1,\dots,y_n\}$, the following subset  is $\delta$-closed in  $\K^n$ and defined over $F$:
$$
\bold{V}(S)=  \left\{ a \in \K^n\:|\: f(a)= 0 \mbox{ for all } f \in S \right\}.$$ 
\end{definition}
\begin{remark}
Since every  radical $\delta$-ideal of  $F\{y_1,  \ldots , y_n\}$ is generated as a radical $\delta$-ideal by a finite set of $\delta$-polynomials  (see, for example, \cite[Theorem, page~10]{RittDiffalg}, \cite[Sections~VII.27-28]{Kapldiffalg}),  the Kolchin topology is {\em Ritt--Noetherian}, that is, every strictly decreasing chain of Kolchin-closed sets has a finite length.
\end{remark}

\begin{definition}
Let $W \subset \K^n$ be a $\delta$-closed set defined over $F$. The $\delta$-coordinate ring $F\{W\}$ of 
$W$ over $F$ is the $F$-$\Delta$-algebra
$$
F\{W\} = F\{y_1,\ldots,y_n\}\big/\I(W).
$$
If $F\{W\}$ is an integral domain, then $W$ is said to be {\it irreducible}. This  is equivalent to $\I(W)$ being a prime $\delta$-ideal.
\end{definition}

\begin{definition}
Let $W \subset \K^n$ be a $\delta$-closed set defined over $F$. 
Let  $\I(W) = \mathfrak{p}_1\cap\ldots\cap \mathfrak{p}_q$ be  a minimal $\delta$-prime decomposition of\ \ $\I(W)$, that is, the $\mathfrak{p}_i \subset F\{y_1,\dots,y_n\}$ are prime $\delta$-ideals containing $ \I(W)$ and minimal with this property. This decomposition is unique  up to permutation (see \cite[Section~VII.29]{Kapldiffalg}).  The irreducible Kolchin-closed sets 
$W_i=\bold{V}(\mathfrak{p}_i)$ are defined over $F$ and  called the {\it irreducible components} of $W$. We  have $W = W_1\cup\ldots\cup W_q$. 
\end{definition}

\begin{definition}
Let $W_1 \subset \K^{n_1}$ and $W_2 \subset \K^{n_2}$ be two Kolchin-closed sets defined over $F$. 
A $\delta$-polynomial map (morphism) defined over $F$ is a map  $$\varphi : W_1\to W_2,\quad a \mapsto \left(f_1(a),\dots,f_{n_2}(a)\right),\ \  a \in W_1,$$ where $f_i \in F\{y_1,\dots,y_{n_1}\}$ for all $i=1,\dots,n_2$.  

If $W_1 \subset W_2$, the inclusion map of $W_1$ in $W_2$ is a $\delta$-polynomial map. In this case, we say that $W_1$ is 
a $\delta$-closed subset of $W_2$.
\end{definition}

\begin{example}
Let $\GL_n \subset \K^{n^2}$ be the group of $n \times n$ invertible matrices   with entries in $\K$. One can see
$\GL_n$ as a Kolchin-closed subset of $\K^{n^2} \times \K$ defined over $F$, defined by the equation $\det(X)y-1$ in $F\big\{\K^{n^2} \times \K\big\}=F\{X,y\}$, where $X$ is an $n \times n$-matrix of $\delta$-indeterminates over $F$ and $y$ a $\delta$-indeterminate over $F$. One can thus identify the $\delta$-coordinate ring of $\GL_n$ over $F$ with  $F\{Y,1/\det(Y)\}$, where $Y=(y_{i,j})_{1 \leq i,j \leq n} $ is
a matrix of $\delta$-indeterminates over $F$. We also denote  the special linear group that consists of the matrices  of determinant $1$ by $\SL_n \subset \GL_n$.

Similarly, if $V$ is a finite-dimensional $F$-vector space,  $\GL(V)$ is defined as the group of invertible   $\K$-linear maps of $V\otimes_F \K$. To simplify the terminology, we will also treat 
$\GL(V)$ as Kolchin-closed sets tacitly assuming that some basis of $V$ over $F$ is fixed. 
\end{example}

\begin{remark}
If  $K$ is a field, we denote  the group of invertible matrices with coefficients in $K$ by  $\GL_n(K)$. 
\end{remark}
\subsection{Differentail algebraic groups}
\begin{definition}A linear differential algebraic group $G \subset \K^{n^2}$ defined  over $F$
is a subgroup of   $\GL_n$ that is a Kolchin-closed set defined over $F$. If $G \subset H \subset \GL_n$ are Kolchin-closed subgroups of 
$\GL_n$, we say that $G$ is a $\delta$-closed subgroup, or $\delta$-subgroup of $H$.
\end{definition}

We will use the abbreviation LDAG for a linear differential algebraic group.

\begin{proposition}\label{propo:defzariksidenserelationalggroupanddiffalggroup}
Let $G \subset \GL_n$ be a linear algebraic group defined over $F$.
\begin{enumerate}[leftmargin=0.5cm,itemindent=0cm,labelsep=0.55cm,itemsep=0.1cm,align=parleft]
\item $G$ is an LDAG. 
\item Let $H \subset G$ be a $\delta$-subgroup of $G$ defined over $F$, and the Zariski closure $\overline{H} \subset G$ be the  closure of $H$
with respect to the Zariski topology. In this case, $\overline{H}$ is a linear algebraic group defined over $F$, whose polynomial defining ideal over $F$ is 
 $$\I(H) \cap F[Y] \subset \I(H) \subset  F\{Y\},$$ 
 where  $Y=(y_{i,j})_{1 \leq i,j \leq n} $ is
a matrix of $\delta$-indeterminates over $F$.
\end{enumerate}
\end{proposition}

\begin{definition}
Let $G$ be an LDAG  defined over $F$. The irreducible component of $G$ containing the identity element $e$  is called the {\it identity component} of $G$ and denoted by $G^\circ$. The LDAG $G^\circ$ is a $\delta$-subgroup of $G$ defined over $F$. 
\end{definition}
\begin{definition}
An LDAG 
$G$ is said to be {\it connected} if $G = G^\circ$, which is equivalent to $G$ being an irreducible Kolchin-closed set \cite[page~906]{cassdiffgr}.
\end{definition}

\begin{definition}Let $G$ be an LDAG defined over $F$ and  $V$   a
finite-dimensional vector space over $F$. A  $\delta$-polynomial
group homomorphism  $\rho : G \to \GL(V)$ defined over $F$ is called a
{\it representation} of $G$ over $F$. 

We shall also say that $V$ is a \emph{$G$-module} over $F$. By a faithful (respectively, simple,  semisimple) $G$-module, we mean
a faithful (respectively, irreducible, completely reducible) representation $\rho:G \rightarrow \GL(V)$.
\end{definition}

The image of a $\delta$-polynomial group homomorphism $\varrho : G\to H$ is Kolchin closed \cite[Proposition~7]{cassdiffgr}. Moreover, if $\ker(\varrho)=\{1\}$, then $\rho$ is an isomorphism of linear differential algebraic groups between $G$ and $\rho(G)$ \cite[Proposition~8]{cassdiffgr}.

\begin{definition}[{\cite[Theorem~2]{Cassunip}}] An LDAG $G$ is {\it unipotent} if one of 
the following equivalent conditions holds:
\begin{enumerate}[leftmargin=0.5cm,itemindent=0cm,labelsep=0.55cm,itemsep=0.1cm,align=parleft]
\item $G$ is conjugate to a differential algebraic subgroup  of the group  of unipotent upper triangular matrices;
\item $G$ contains no elements of finite order $>1$;
\item $G$ has a descending normal sequence  of differential algebraic subgroups 
$$G=G_0 \supset G_1 \supset \ldots \supset G_N =\{1\}$$
with $G_i/G_{i+1}$ isomorphic to a differential algebraic subgroup of the additive group $\bold{G}_a$.
\end{enumerate}
\end{definition}

One can  show that 
an LDAG   $G$ defined over $F$ admits a  largest normal unipotent differential algebraic subgroup defined over $F$ \cite[Theorem~3.10]{diffreductive}.

\begin{definition} Let $G$ be an LDAG defined over $F$. The  largest normal unipotent differential algebraic subgroup of $G$ defined over $F$ is  called the {\it unipotent radical} of $G$ 
and denoted by $\Ru(G)$. The unipotent radical of a linear algebraic group $H$ is also denoted by $\Ru(H)$.
\end{definition}

\begin{definition}A non-commutative LDAG $G$ is said to be {\em simple} if $\{1\}$ and $G$ are the only normal differential algebraic subgroups of $G$.
\end{definition}

\begin{definition}
A \emph{quasi-simple} linear (differential) algebraic group is a finite central extension of a simple non-commutative linear (differential) algebraic group.
 \end{definition}
 
\begin{definition}
An LDAG $G$  is said to be {\it reductive} if  $\Ru(G) = \{1\}$.
\end{definition}

\begin{proposition}[{\cite[Remark~2.9]{MinOvSing}}]
Let $G \subset \GL_n$ be an LDAG. If $\overline{G} \subset \GL_n$
is a reductive linear algebraic group, then $G$ is a reductive LDAG.
\end{proposition}

For a group $G$, its subgroup generated by $\{ghg^{-1}h^{-1}\:|\: g,h \in G\}$ is denoted by $[G,G]$. 

\begin{definition}
Let $G$ be a group and $G_1,\dots,G_n$  subgroups of $G$. We say that 
$G$ is the almost direct product of $G_1,\dots,G_n$ if 
\begin{enumerate}[leftmargin=0.5cm,itemindent=0cm,labelsep=0.55cm,itemsep=0.1cm,align=parleft]
\item the commutator subgroups $[G_i,G_j]=\{1\}$ for all $i \neq j$;
\item the morphism $$\psi :G_1\times \hdots \times G_n \rightarrow G,\quad (g_1,\dots,g_n) \mapsto g_1\cdot\hdots \cdot g_n$$
is an isogeny, that is, a surjective map with a finite kernel.  
\end{enumerate}
\end{definition}

\begin{theorem}[{\cite[Theorem~2.25]{HMO}}]\label{thm:decompalmostdirectprodreductive}
Let $G \subset \GL_n$ be a linear differential algebraic group defined over $F$. Assume that  
 $\overline{G} \subset \GL_n $ is  a connected reductive algebraic group.
Then
\begin{enumerate}[leftmargin=0.5cm,itemindent=0cm,labelsep=0.55cm,itemsep=0.1cm,align=parleft]
\item\label{part1} $\overline{G}$ is an almost direct product of a torus $H_0$ and non-commutative 
normal quasi-simple linear algebraic groups $H_1,\dots,H_s\,$ defined over $\mathbb{Q}$;
\item\label{part2} $G$ is an almost direct product of a Zariski dense $\delta$-closed subgroup $G_0$ of $H_0$ and some $\delta$-closed subgroups $G_i$ 
of $H_i$ for $i=1,\dots,s$;
 \item\label{part3}  moreover, either $G_i=H_i$ or $G_i$ is conjugate by a matrix of $H_i$ to $H_i(C)$;
\end{enumerate}
\end{theorem} 

\begin{definition}
Let $G$ be an LDAG defined over $F$. We define $\tau(G)$, the {\em differential type} of $G$, to be $0$ if $\trdeg_F\Quot(F\{G^\circ\}) <\infty$ and to be $1$ otherwise.
\end{definition}

 \begin{definition}\label{def:DFGG} Let $G$ be an LDAG defined over $F$.  We say that $G$ is \emph{differentially finitely generated}, or simply a \emph{DFGG}, if $G(\K)$ contains a  finitely generated subgroup that is Kolchin dense over $F$. 
\end{definition}

\subsection{Differential modules and their Galois groups}Our presentation of the (parameterized) differential Galois theory is deliberately based on fiber functors and tensor categories so that the description of our main algorithm is clearer. Indeed,  it is essential for the description to have a correspondence between the operations performed with the differential module (system of linear differential equations) and with representations of its (parameterized) differential Galois group.

Let $K$ be  a $\Delta=\{\p,\delta\}$-field and $\k=K^\p$.
We assume for simplicity that $(\k,\delta)$ is  a differentially closed field (this assumption was relaxed in \cite{GGO,wibdesc,LSN}).
\begin{definition}
  A $\p$-module
$ \cM $ over $K$ is a left $K[\p]$-module that is a finite-dimensional vector space over $K$.
\end{definition}

Let $\cM$ be a $\p$-module over $K$
and $\{e_1,\dots,e_n\}$ a $K$-basis of $\cM$. Let $A=(a_{i,j}) \in \Mat_n(K)$ be the matrix defined
by 
\begin{equation}\label{eq:dm}
\p (e_i)= - \sum_{j=1}^n a_{j,i} e_j,\quad i=1,\dots,n.
\end{equation} Then, for any $m=\sum_{i=1}^n y_i e_i$, where 
$Y=(y_1,
 \ldots ,
 y_n)^T \in K^n$, we have $$\p(m) = \sum_{i=1}^n \p(y_i)e_i - \sum_{i=1}^n \left( \sum_{j=1}^n a_{i,j} y_j\right)e_i.$$ 
 Thus, the equation $\p(m)=0$ translates into the homogeneous system of linear differential equations $\p(Y)=AY$. 
 \begin{definition}\label{defn:systmodule}
Let $\cM$  be a $\p$-module over $K$
and  $\{e_1,\dots,e_n\}$ be a $K$-basis of $\cM$.
 We say that the linear differential
 system $\p(Y)= AY$, as above,  is associated  to the  $\p$-module $\cM$ (via the choice of a $K$-basis).  
 Conversely, to a given  
 linear differential system $\p(Y)=AY$, $A=(a_{i,j}) \in K^{n\times n}$, one associates a $\p$-module
 $\cM$ over $K$, namely $\cM=K^n$ with the standard basis $(e_1,\dots,e_n)$ and action of $\p$ given by~\eqref{eq:dm}.
  \end{definition}
 
  \begin{definition} A
 morphism  of $\p$-modules over $K$ is  a  homomorphism of $K[\p]$-modules. 
 \end{definition}

One can  consider  the category  $\Diff_K$ of $\p$-modules over $K$:

 \begin{definition}
We can define the following constructions in $\Diff_K$:
\begin{enumerate}[leftmargin=0.5cm,itemindent=0cm,labelsep=0.55cm,itemsep=0.1cm,align=parleft]
\item The direct sum of two $\p$-modules, $\cM_1$ and $\cM_2$, is $\cM_1 \oplus \cM_2$ together with the action of
$\p$ defined by $$\p(m_1\oplus m_2)= \p(m_1)\oplus \p( m_2).$$ 
\item The tensor product of two $\p$-modules, $\cM_1$ and $\cM_2$, is $\cM_1 \otimes_K \cM_2$ together with the action of
$\p$ defined by $$\p(m_1 \otimes m_2)= \p(m_1)\otimes m_2 +m_1 \otimes \p( m_2).$$ 
\item
The unit object $\bold{1}$ for the tensor product  is
the field $K$ together with  the left $K[\p]$-module structure given by $$(a_0+ a_1 \p +\dots +a_n\p^n)(f)= a_0f+\dots+ a_n \p^n(f)$$ for $f,a_0,\dots,a_n \in K$. 
\item The internal Hom of  two $\p$-modules $\cM_1, \cM_2$ exists in $\Diff_K$ and 
is denoted by $\underline{\Hom}(\cM_1,\cM_2)$. It consists of the $K$-vector space $\Hom_K(\cM_1,\cM_2)$
of  $K$-linear maps from $\cM_1$ to $\cM_2$ together with the action of $\p$ given by the formula 
$$\p u(m_1)= \p(u(m_1))-u(\p m_1).$$ The dual $\cM^ *$ of a $\p$-module $\cM$ is the $\p$-module $\underline{\Hom}(\cM,\bold{1})$.
\item An endofunctor  $D : \Diff_K\to \Diff_K$, called the prolongation functor, is defined as follows: 
 if $\cM$ is an object of $\Diff_K$ corresponding to the
linear differential system $\p (Y)=A Y$, then $D(\cM)$ corresponds to the linear differential system $$\p(Z)=\begin{pmatrix} A & \delta(A) \\ 0& A \end{pmatrix} Z.$$
\end{enumerate}
\end{definition}

The construction of the prolongation functor  reflects
the following. If $U$ is a fundamental solution matrix of $\p(Y)=AY$
in some $\Delta$-field extension  $F$ of $K$, that is, $\p(U)=AU$ and $U \in \GL_n(F)$, then
$$
 \p(\delta U) =\delta(\p U)=\delta(A) U +A \delta (U).
$$
Then, $\begin{pmatrix} U & \delta(U) \\ 0& U \end{pmatrix}$ is a fundamental solution matrix of $\p(Z)=\begin{pmatrix} A & \delta(A) \\ 0& A \end{pmatrix} Z$.
 Endowed with all these constructions, it follows from \cite[Corollary~3]{OvchTannakian} that the category $\Diff_K$ is  a {\em $\delta$-tensor category} (in the sense of \cite[Definition~3]{OvchTannakian} and \cite[Definition~4.2.1]{MosheTAMS}).

\begin{definition}
 Let $\cM$ be an object of $\Diff_K$. Let $\{\cM\}^{\otimes, \delta}$ denote
 the smallest full subcategory of $\Diff_K$ that 
contains $\cM$ and is closed under all operations of linear algebra (direct sums, tensor products, duals, and subquotients)
and $D$. The category  $\{\cM\}^{\otimes, \delta}$ is  a $\delta$-tensor category over~$\k$. Let $\{\cM\}^{\otimes}$ denote the full tensor 
subcategory of $\Diff_K$ generated by $\cM$.   Then, $\{\cM\}^{\otimes}$ is a tensor category over~$\k$.
\end{definition}

Similarly, the category $\Vect_\k$ of finite-dimensional 
$\k$-vector spaces is a $\delta$-tensor category.  The prolongation functor on $\Vect_\k$ is defined as follows: for  a $\k$-vector space
$V$, the $\k$-vector space $D(V)$  equals  $\k[\delta]_{\leq 1}\otimes _\k V$, where $\k[\delta]_{\leq 1}$ is considered as the right $\k$-module
of $\delta$-operators up to order $1$ and $V$ is viewed as a left $\k$-module.

\begin{definition}\label{defn:fiberfunctor}
Let  $\cM$ be an  object of $\Diff_K$.  A $\delta$-fiber functor $\o :\{\cM\}^{\otimes, \delta} \rightarrow \Vect_\k$ is
an exact, faithful, $\k$-linear, tensor compatible functor  together with a natural isomorphism  between $D_{\Vect_\k} \circ \omega$ and 
$\omega\circ D_{ \{\cM\}^{\otimes, \delta}}$ \cite[Definition~4.2.7]{MosheTAMS}, where the subscripts emphasize the category on which  we perform the prolongation. The pair $\big(\{\cM\}^{\otimes, \delta}, \omega\big)$ is called a $\delta$-Tannakian category.
\end{definition}

\begin{theorem}[{\cite[Corollaries~4.29 and~6.2]{GGO}}]
Let $\cM$ be an object of $\Diff_K$.
Since $\k$ is $\delta$-closed,  the category $\{\cM\}^{\otimes, \delta}$ admits a $\delta$-fiber functor and 
 any two 
 $\delta$-fiber functors are naturally isomorphic.
\end{theorem}

\begin{definition}\label{defn:difftannakautogroup}
Let $\cM$ be an object of $\Diff_\k$ and $\o :\{\cM\}^{\otimes, \delta} \rightarrow \Vect_\k$ be a $\delta$-fiber functor. 
The group $\Gal^\delta(\cM)$ of $\delta$-tensor isomorphisms  of $\omega$ is defined as follows. It consists
of  the elements $ g\in\GL(\o(\cM))$ that stabilize $\omega(\mathcal{V})$ for  every $\p$-module  $\mathcal{V}$ obtained from $\cM$ by applying the linear constructions (subquotient, direct sum, tensor product, and dual), and the 
prolongation functor. The action of $g$ on $\omega(\cV)$ is obtained by applying the same constructions to $\omega(\cV)$. We call $\Gal^\delta(\cM)$ the {\em parameterized differential Galois group} of $(\cM,\omega)$ (or of $\cM$ when there is no confusion).\end{definition}

\begin{theorem}[{\cite[Theorem~2]{OvchTannakian}}]\label{thm:difftanequ}
Let $\cM$ be an object of $\Diff_K$ and $\o :\{\cM\}^{\otimes, \delta} \rightarrow \Vect_\k$ be a $\delta$-fiber functor. The group  $\Gal^\delta(\cM) \subset \GL(\o(\cM))$ is a linear differential algebraic group defined over $\k$,  and $\o$ induces an equivalence of  categories between $\{\cM\}^{\otimes, \delta}$
and the category of finite-dimensional  representations of $\Gal^\delta(\cM)$.
\end{theorem}
\begin{definition}\label{defn:trivial objects}
We say that a $\p$-module $\cM$ over $K$ is \textit{trivial} if it is either $(0)$ or isomorphic as a $\p$-module over $K$ to  $\bold{1}^n$ for 
some positive integer $n$. For $G$ a linear differential algebraic group over $\k$, we say that 
a $G$-module $V$ is \textit{trivial} if   $G$ acts identically  on $V$. 
\end{definition}
\begin{remark}
For $\cM$ an object of $\Diff_K$ and $\o :\{\cM\}^{\otimes, \delta} \rightarrow \Vect_\k$  a $\delta$-fiber functor, the following holds:
a $\p$-module $\cN $ in $\{\cM\}^{\otimes, \delta}$ is trivial if and only if $\o(\cN)$ is a trivial $\Gal^\delta(\cM)$-module.
\end{remark}

\begin{definition}
Forgetting the action of $\delta$,  one can similarly define the group $\Gal(\cM)$ of tensor isomorphisms of $\omega : \{\cM\}^{\otimes} \rightarrow \Vect_\k$. By \cite{Deligne:categoriestannakien}, the group $\Gal(\cM) \subset \GL(\o(\cM))$ is a linear algebraic group defined over $\k$, and $\omega$ induces an equivalence of categories between $\{\cM\}^{\otimes}$ and the 
category of  $\k$-finite-dimensional representations of $\Gal(\cM)$. We call  $\Gal(\cM)$ the {\em differential Galois group} of 
$\cM$ over~$K$.
\end{definition}

\begin{proposition}[{\cite[Proposition~6.21]{HardouinSinger}}]\label{prop:zariskidensegalois groups}
If $\cM$ is an object of $\Diff_K$ and $\o :\{\cM\}^{\otimes, \delta} \rightarrow \Vect_\k$ is a $\delta$-fiber functor, then $\Gal^\delta(\cM)$ is a Zariski dense subgroup of $\Gal(\cM)$.
\end{proposition}

\begin{definition}
A $\p$-module $\cM$ is said to be {\em completely reducible} ({\em semisimple}) if, for every $\p$-submodule $\cN$ of $\cM$, there exists
a $\p$-submodule $\cN'$ of $\cM$ such that $\cM=\cN \oplus \cN'$. We say that a $\p$-operator is completely reducible if 
the associated $\p$-module is completely reducible.  
\end{definition}
By \cite[Exercise 2.38]{vdPutSingerDifferential},
a $\p$-module is completely reducible if and only if its  differential Galois group is a reductive linear algebraic group.
Moreover, for  a completely reducible $\p$-module $\cM$, any object in $\{\cM\}^\otimes$ is completely
reducible.

\begin{definition}
A linear differential algebraic group $G\subset\Gl(V)$ is called \emph{constant} (cf. \cite[Definition~4.1]{GO}, and also \cite[Definition~2.10]{MinOvSingunip} for reductive LDAGs) if $G$ preserves a $C$-form of $V$. That is, if one considers $G$ as a group of matrices, $G$ is conjugate to a group of matrices with constant entries. 
\end{definition}

\begin{definition}\label{def:purconst} We say that a differential module $\cV$ is {\em constant} (also known as isomonodromic or completely integrable) if $\Galdelta(\cV)$ is constant (see \cite[Proposition~3.9]{cassisinger}, \cite[Section~6]{GO}). We call a semisimple differential module \emph{purely non-constant} if every its simple submodule is not constant.\end{definition}

In the language of matrices, if $\p(Y)=AY$ is a system of linear differential equations associated to $\cV$, then $\cV$ is constant if and only if there exists an $n\times n$ matrix $B$ with entries in $K$ such that $\p(B)-\delta(A)=AB-BA$.

\section{The algorithm}\label{sec:Algorithm}
 For simplicity, let $(\k,\delta)$ be a differentially closed field. Let also $K = \k(x)$ be a $\Delta$-field, where $\Delta=\{\p,\delta\}$, and $\p=\partial/\partial_x$. So, $\k = K^\p$.
Let $\cV$ be a $\p$-module over  $K$, $\dim\cV = 3$, $\omega$ a $\delta$-fiber functor, and so $V=\omega(\cV)$ is a faithful $\Galdelta(\cV)$-module (to avoid repetition, $\p$-modules over $K$ will be denoted by calligraphic letters, and the corresponding regular letters will be used for their images under $\omega$).

In what follows, we start with preliminary results our algorithm-specific definitions in Section~\ref{sec:prep}, where we also explain what is sufficient to know to consider $\Galdelta(\cV)$ computed. We continue with a description of the algorithm by considering the following cases that can possibly occur:
\begin{enumerate}[leftmargin=0.5cm,itemindent=0cm,labelsep=-0.5cm,itemsep=0.1cm,align=left]
\item Optimized for $\dim\cV = 3$, and so more efficiently than in~\cite{MinOvSing}, we start with computing $\Galdelta(\cV^{\diag})$ (see Definition~\ref{def:diag}) in Section~\ref{sec:compdiag}, which also covers the case of $\cV \cong \cV^{\diag}$, that is $\cV$ being semisimple. For checking the latter isomorphism, see  \cite[Section~4.2]{vdPutSingerDifferential} and the references given there.
\item We then proceed with computing $\Galdelta(\cV)$ if $\cV$ is decomposable (can be represented as a direct sum of a  $2$-dimensional not semisimple and a $1$-dimensional  submodules)  in Section~\ref{sec:decompcase}.
\item If $\cV$ is indecomposable and $\cV^{\diag}$ has a $2$-dimensional simple submodule, then we compute $\Galdelta(\cV)$ in Section~\ref{sec:indeccase2dim}.
\item The remaining case of indecomposable $\cV$ with $\cV^{\diag}$ being a direct sum of three $1$-dimensional submodules in Section~\ref{sec:idecuppertriang}.
\end{enumerate}

\subsection{Preparation}\label{sec:prep}
If $G=\Galdelta(\cV)$, as explained above, it can be  identified with an LDAG in $\Gl(V)$, where $V=\o(\mathcal{V})$. To compute $G$ means to provide an algorithm that, given $\cV$, returns a (finite) set of equations defining $G$ in the ring of differential polynomials in matrix coefficients, with respect to some basis of~$V$. Note that $\overline{G}$ can be computed as it is $\Gal(\cV)$  \cite{Hrushcomp,RFdifferential2015}. If one knows a priori that $\tau(G)=0$, one can compute $G$ using \cite{MinOvSingunip}.

Our goal now is to explain an approach to computing $G$ in the case $\dim\cV=3$. By \cite{Dreyfus:density,MitSingMonod}, since $K=\k(x)$,  $G$ is a DFGG, which will be essential. For example, if $\varrho: G\to \Gl(W)$ is a 1-dimensional representation of $G$, then $\tau(\varrho(G))=0$.   Note that the case $\dim\cV= 2$ has already been considered in \cite{CarlosISSAC,ArrecheJSC,Dreyfus}.

Our approach is related to the following. Let $(H_i,V_i)$, $i\in I$, be pairs of algebraic groups $H_i\subset\Gl(V)$ containing $G$, and algebraic $H_i$-modules $V_i$, where $I$ is a non-empty finite set. Let $\varrho_i:H_i\to\Gl(V_i)$ denote the homomorphisms defining the $H_i$-module structure on $V_i$. 
\begin{definition}\label{def:determined}
We say that $G$ is
\begin{itemize}[leftmargin=0.35cm,itemindent=0cm,labelsep=0.3cm,itemsep=0.1cm,align=parleft]
\item \emph{determined by the pairs} $(H_i,V_i)$, $i\in I$, if $G$ has finite index in the intersection $P$ of all $\varrho_i^{-1}(G_i)$, where $G_i:=\varrho_i(G)$. In other words, $G^\circ=P^\circ$, or equivalently, due to the bijection between the finite sets of connected components of $G$ and of $\overline{G}$~\cite[Corollary~3.7]{MinOvSing},
$$
G=P\cap\overline{G}.
$$
\item
 {\em determined by} $V_i$ if $H_i=\overline{G}$, $i\in I$. In this case, $G=P$.
\end{itemize}
\end{definition}
\begin{remark}
 It follows that $G$ is determined by $V_i$ if it is determined by some $(H_i, V_i)$.
 \end{remark}
 \begin{remark}
  Note that $V_i$ are $G$-modules from the rigid tensor category generated by $V$ and, if $H_i=\overline{G}$, $i\in I$, this property characterizes them. 
  \end{remark}
  This notion can be applied for computation of $G$. Namely, suppose it is known that $G$ is determined by $(H_i,V_i)$ and suppose that we know how to compute $G_i$. Then we can compute
$$
G=\bigcap_{i\in I}\varrho_i^{-1}(G_i)\cap\overline{G}
$$ 
since $\varrho_i$ are given (the intersection with $\overline{G}$ is needed only if $G\neq P$). Furthermore, if $F_i\subset \k\{\Gl(V_i)\}$, $J_i, J\subset \k\{\Gl(V)\}$,  $i\in I$, are sets of generators of the defining ideals of $G_i$, $H_i$, and $H$, respectively, then the defining ideal for $G$ is generated by
$$
\bigcup_{i\in I}\nu\varrho_i^*(F_i)\cup J_i\cup J,
$$
where $\nu$ is a $k$-linear section of $\k\{\Gl(V)\}\to \k\{H_i\}$. Again, union with $J$ is only needed if we do not know whether $G=P$. 

\begin{proposition}\label{prop:index}
Let $G_i\subset P_i$, $i=1,2$, be subgroups of finite index and all $P_i$ are subgroups of a group~$H$. Then $G_1\cap G_2\subset P_1\cap P_2$ has finite index.
\end{proposition}
\begin{proof}
The groups $P_i$ act naturally on the finite sets $X_i:=P_i/G_i$, $i=1,2$. This gives rise to the action of $P_1\times P_2$ on $X:=X_1\times X_2$. The group $P:=P_1\cap P_2$ embedded diagonally into $P_1\times P_2$ also acts on $X$ with the stabilizer of the point $eG_1\times eG_2\in X$ equal $G:=G_1\cap G_2$. Since $X$ is finite, $G$ has finite index in $P$.
\end{proof}

\begin{definition}\label{def:diag}
Let us denote  the sum of composition factors for a maximal filtration of $\cV$ by $\cV^{\diag}$. Set also $V^{\diag}:=\o(\cV^{\diag})$.  
\end{definition}

One can compute $\Galdelta(\cV^{\diag})$ using \bl{\cite{MinOvSing}}. However, this can be done simpler in our case of three dimensions.

\begin{lemma}[{\cite[Equation (1), p.~195]{CassSingerJordan}}]\label{lem:inequtype}
Let   $G$  be a linear differential algebraic group and $H$ be  a normal differential algebraic subgroup of $G$. Then $\tau(G) = \max\{\tau(H),\tau(G/H)\}$.\end{lemma}

\begin{definition}[{\cite[Definition~2.6]{CassSingerJordan}}] The {\em strong identity component} $G_0$ of an LDAG $G$ is defined to be the smallest $\delta$-subgroup $H$ of $G$ such that $\tau(G/H) < \tau(G)$. 
An LDAG $G$ is called {\em strongly connected} if $G_0 = G$.
\end{definition}
\begin{lemma}\label{lem:strongidepi}If $G$ and $G'$ are LDAGs, $\varphi : G \to G'$ is a surjective homomorphism, and $\tau(G)=\tau(G')$, then $\varphi(G_0)=G'_0$.
\end{lemma}
\begin{proof}By \cite[Remark~2.7.4]{CassSingerJordan}, $\varphi(G_0) \subset G'_0$. On the other hand,  using Lemma~\ref{lem:inequtype},
$$\tau(G'/\varphi(G_0)) = \tau (G/G_0\ker\varphi)\leq \tau(G/G_0) < \tau(G)=\tau(G'),$$
which, by definition, implies that $\varphi(G_0) \supset G'_0$.
\end{proof}

\subsection{Computation of the diagonal}\label{sec:compdiag}
Here, we will explain how to find $G = \Galdelta(\cV^{\diag})$. Without loss of generality, in this section, let us assume that 
\begin{equation}\label{eq:isomdiag}
\cV\cong\cV^{\diag},
\end{equation}
 so we need to compute $G$. In this case, $G$ is reductive, so  \begin{equation}\label{eq:SZ}G^\circ =S\cdot Z,\end{equation} an almost direct product of its center $Z$ and a semisimple $\delta$-subgroup $S\subset\Sl(V)$. Due to the restriction $\dim\cV=3$, one can see that $S$ has to be simple, so it is $\delta$-isomorphic to a simple algebraic group or its constant points.

\begin{proposition}\label{prop:commutative}
Let $G$ be a DFGG such that $G^\circ$ is commutative. Then $\tau(G)=0$.
\end{proposition}
\begin{proof}
By \cite[Proposition~2.11]{MinOvSing}, $G^\circ$ is a DFGG. Let $F\subset\GL(V)$ denote the Zariski closure of $G^\circ$. By~\cite[Theorem 15.5]{Humphline}, $F$ is a direct product of an algebraic torus $F_s$ and a commutative unipotent group $F_u$. Let $G_s$ and $G_u$ stand for the projections of $G^\circ$ to $F_s$ and $F_u$, respectively. Both $G_s$ and $G_u$ are DFGG as  homomorphic images of a DFGG. It follows from \cite[Lemmas 2.12 and 2.13]{MinOvSing} that $\tau(G_s)=\tau(G_u)=0$. Since $G^\circ\subset G_s\times G_u$, $\tau(G^\circ)=0$. Hence, $\tau(G)=0$.
\end{proof}

\begin{proposition}\label{prop:semisimple}
Suppose that $G^\circ$ is non-commutative.
\begin{enumerate}[leftmargin=0.5cm,itemindent=0cm,labelsep=0.55cm,itemsep=0.1cm,align=parleft]
\item If $\cV$ is simple, then $G$ is determined by $(\GL(V),V\otimes V^*)$ and $(\GL(V),\wedge^{3}V)$.
\item If $\cV$ is not simple, then $\cV=\cU\oplus\cW$, where $\dim\cU=1$ and $\cW$ is simple, and $G$ is determined by $(H,W\otimes W^*)$ and $(H,\wedge^2W\oplus U)$, where $W = \omega(\cW)$, $U = \omega(\cU)$, and $H=\GL(U)\times\GL(W)\subset\GL(V)$.
\end{enumerate}
\end{proposition}

\begin{proof}
By \cite[Proposition 2.23]{HMO}, $V=\o(\cV)$ is semisimple as a $G^\circ$-module.	
Suppose that $V$ is simple. Then $V$ is simple as a $G^\circ$-module. Indeed, since $G^\circ$ is non-commutative, $V$ has an irreducible $G^\circ$-submodule $W$ of dimension $\geq 2$. If $\dim W=2$, then, due to $\dim V=3$, $W$ is the only irreducible submodule of $V$ of dimension~2. Since $G^\circ$ is normal in $G$, it follows that $W$ is a $G$-submodule of $V$, which is impossible. Hence, $\dim W=3$ and $V$ is simple as a $G^\circ$-module. 
Set $$H=H_1=H_2=\Gl(V),\ \ V_1:=V\otimes V^*\simeq\End(V),\ \ V_2:=\wedge^{3}V.$$ 
Recall the notation from Section~\ref{sec:prep}: $\varrho_i := \overline{G} \to \GL(V_i)$, $i=1,2$, where the Zariski closure is taken in $\GL(V)$.
We have $\Ker\varrho_1\subset H$ equal the group $E$ of scalar multiplications (by Schur's lemma) and  $\Ker\varrho_2=\Sl(V)$.
It is sufficient to show (see Definition~\ref{def:determined}) that $G$ has finite index in $$\Gamma:=(G\cdot E)\cap (G\cdot\SL(V))\subset\GL(V).$$ By Proposition~\ref{prop:index}, $\Gamma$ contains the finite index subgroup $$\Gamma_1:=(G^\circ\cdot E)\cap (G^\circ\cdot\SL(V))=(S\cdot E)\cap (Z\cdot\SL(V)).$$
(Here we used~\eqref{eq:SZ}, $S\subset\SL(V)$,  because $S=[S,S]$, and $Z\subset E$ by Schur's lemma, because $V$ is a simple $G^\circ$-module, as shown above.) Setting $\Lambda=\SL(V)\cap E$, which is finite, we obtain $G^\circ\subset \Gamma_1= G^\circ\cdot\Lambda$. Indeed, the inclusion $ \Gamma_1\supset G^\circ\cdot\Lambda$ is straightforward. 
For the reverse inclusion, for each $\gamma\in \Gamma_1$, let $s \in S$, $e\in E$, $z\in Z$, and $a \in \SL(V)$ be such that $\gamma=se=za$. Since $z$ is a scalar matrix, $z^{-1}e=s^{-1}a\in \SL(V)$, and so there exists $\lambda \in\Lambda$ such that $e=z\lambda$. Therefore, $\gamma =s(z\lambda)=(sz)\lambda \in G^\circ\cdot\Lambda$.
Hence, $G^\circ$ has finite index in $\Gamma_1$, and therefore in $\Gamma$. Hence, $G$ has finite index in $\Gamma$.

If $\cV$ is not simple, then there exist $\cU$ and $\cW$ such that $\dim\cU=1$, $\dim\cW =2$, and $\cV=\cU\oplus\cW$. 
Let $\varrho_1 : H \to \GL(W\otimes W^*)$ and $\varrho_2 : H \to \GL(\wedge^2W\oplus U)$, where
$$H=H_1=H_2=\Gl(U)\times\Gl(W)\subset\Gl(V).$$ $\Ker\varrho_1$ equals the group  of transformations acting by scalar multiplications on $U$ and $W$, and $\Ker\varrho_2=\Sl(W)$. Since $S=[S,S]$, $S \subset \SL(W)$. By Schur's lemma, $Z \subset \Ker\varrho_1$. Therefore, 
$$\Gamma_1' := (G^\circ\cdot \ker\varrho_1)\cap (G^\circ\cdot\ker\varrho_2)=(S\cdot\ker\varrho_1)\cap (Z\cdot\SL(W)),$$
which, by Proposition~\ref{prop:index}, has finite index in $$\Gamma' := (G\cdot \ker\varrho_1)\cap (G\cdot\ker\varrho_2)\subset\GL(U)\times\GL(W).$$
Let $\Lambda' =\ker\varrho_1\cap \SL(W)$, which is finite.
We have $G^\circ\cdot\Lambda' \subset \Gamma_1'$.  For each $\gamma\in \Gamma_1'$, let $s \in S$, $e\in \ker\varrho_1$, $z\in Z$, and $a \in \SL(W)$ be such that $\gamma=se=za$. Since $z$ and $s$ commute, $z^{-1}e=s^{-1}a\in \SL(W)$, and so there exists $\lambda \in\Lambda'$ such that $e=z\lambda$. Therefore, $\gamma =s(z\lambda)=(sz)\lambda \in G^\circ\cdot\Lambda'$, and so $\Gamma_1'\subset G^\circ\cdot\Lambda'$.
Hence, $G^\circ$ has finite index in $\Gamma_1'$, and therefore in $\Gamma'$. Hence, $G$ has finite index in $\Gamma'$.
\end{proof}

\begin{remark}\label{rem:semisimple}
In the notation of Proposition~\ref{prop:semisimple}, the groups $\Galdelta(\wedge^3\cV)$ and $\Galdelta(\wedge^2\cW\oplus\cU)$ can be computed since their differential type is $0$ by Proposition~\ref{prop:commutative}. Moreover,  $\Galdelta(\cV\otimes\cV^*)$ and $\Galdelta(\cW\otimes\cW^*)$ are quasi-simple, so can be computed using \cite[Algorithm~5.3.2]{MinOvSing}.  
\end{remark}
\subsection{Decomposable case}\label{sec:decompcase}
Let us consider the situation in which $\cV$ is a direct sum of its nontrivial submodules $\cU$ and $\cW$, and $\dim\cU=1$. (One of them has to be of dimension 1.) The case of semisimple $\cV$ (or, equivalently, $\cW$) is treated by Propositions~\ref{prop:commutative} and~\ref{prop:semisimple}. So, we are interested in what happens if $\cW$ is not a direct sum.

We will use the following result for differential modules of dimension $2$ (see \cite[Theorem~2.13]{MinOvSingunip}, \cite[Proposition~3.21]{HMO}, and \cite{ArrecheJSC}).

\begin{proposition}\label{prop:2solvable}
Let $G'$ be an LDAG and $W$ a 2-dimensional faithful $G'$-module such that $W^{\diag}=U_1\oplus U_2$, $\dim U_i=1$. Then one of the following holds:
\begin{enumerate}[leftmargin=0.95cm,itemindent=0cm,labelsep=0.95cm,itemsep=-0.35cm,align=parleft]
 \item[(CQ)] $U_1\otimes U_2^*$ is constant and $\tau(G')$=0;\\
 \item[(CR)] $U_1\otimes U_2^*$ is non-constant and $W$ is semisimple;\\
 \item[(NC)] $U_1\otimes U_2^*$ is non-constant and $\Ru(G')\simeq\Ga$.\\
\end{enumerate}
\end{proposition}

\begin{proposition}\label{prop:decomposable}
Suppose $\cV=\cU\oplus\cW$, $\dim\cU=1$, $\dim\cW=2$, and $\cW$ is not semisimple. Let $\cW^{\diag}=\cW_1\oplus\cW_2$.
\begin{enumerate}[leftmargin=0.1cm,itemindent=0cm,labelsep=.15cm,itemsep=-0.3cm,align=left]
 \item If $\Galdelta(\cW_1\otimes\cW_2^*)$ is constant, then $\tau(G)=0$.\\
 \item Otherwise, $G$ is determined by $V^{\diag}$.
\end{enumerate}
\end{proposition}

\begin{proof}
The first case follows from Proposition \ref{prop:2solvable} and Lemma~\ref{lem:inequtype}. For the second case, let $W_0\subset W$ be a $G$-invariant 1-dimensional subspace, and define $$H\subset\Gl(U)\times\Gl(W)\subset\Gl(V)$$ to be the group of transformations preserving $W_0$. We have $G\subset H$. We claim that $G$ is determined by $(H,V^{\diag})$. It suffices to show that $G$ contains the kernel $N  \cong \Ga$ of the action of $H$ on $V^{\diag}$. We have $$N\subset\Gl(W)\subset\Gl(V).$$
Let $$\varphi: \Gl(U)\times \Gl(W) \to \Gl(W)$$ be the projection homomorphosm and $G_W = \varphi(G)$. By the hypothesis and Proposition \ref{prop:2solvable}, $G_W\supset N$. Since $N \cong \Ga$, it is strongly connected~\cite[Example 2.7(1)]{CassSingerJordan}. Moreover, since $G_W/N$ is diagonal and is a DFGG, $$\tau(G_W/N) = 0 < \tau(G_W)=1.$$ Therefore, by definition, $N=(G_W)_0$. Hence, $N=\varphi(G_0)$. Since $G$ is a DFGG, its restriction to $\Gl(U)$ is of type zero and, therefore, the image of $G_0$ in $\Gl(U)$ is of type zero as well. Moreover,  
by \cite[Remark~2.7.5]{CassSingerJordan},
we conclude that $G_0$ acts trivially on $U$. Therefore, since $G_0\subset\Gl(U)\times N$, and $G_0$ acts trivially on $V^{\diag}$, 
we conclude $G_0\subset N$, hence $N=G_0\subset G$.
\end{proof}

\subsection{Indecomposable case with a 2-dimensional simple component}\label{sec:indeccase2dim}
\begin{proposition}\label{prop:indecomposable}
Suppose that $\cV$ is indecomposable and $\cV^{\diag}$ is the sum of two simple submodules $\cW_1$ and $\cW_2$, of dimension $1$ and $2$, respectively. Suppose additionally that $G^\circ$ is not commutative. If $\Galdelta(\cW_1^*\otimes\cW_2)$ is constant, then $\tau(G)=0$. If $\Galdelta(\cW_1^*\otimes\cW_2)$ is not constant, then $G$ is detemined by $\omega(\cV)^{\diag}$. 
\end{proposition}

\begin{proof}
Since the image of $G$ under  $\rho : G \to \Gl(W_1)$ is a DFGG, it is of type zero.
If $\cW:=\cW_1^*\otimes\cW_2$ is constant, so is $(\cW_1^*\otimes\cV)^{\diag}$. Hence, the image of $G$ under $\varphi : G \to \Gl(W_1^*\otimes V)$ is of type zero by \cite[Proposition 2.19]{MinOvSingunip}. Since $\ker\varphi$ consists of scalar transformations, it is isomorphic to a subgroup of $\rho(G)$, therefore $\tau(\ker\varphi)=0$. We conclude by Lemma~\ref{lem:inequtype} that $$\tau(G) = \max\{\tau(\ker\varphi),\tau(\varphi(G))\}=0.$$ Suppose now that $\cW$ is non-constant. Replacing $\cV$ with $\cV^*$ if needed, we obtain $$\cW_2\subset\cV.$$ 
By \cite[Proposition~3.12]{HMO}, there is a $G$-equivariant embedding of $\Ru(G)$ into $W$.  Since $W$ is irreducible and non-constant, it does not have nonzero proper invariant $\delta$-subgroups by \cite[Proposition 3.21]{HMO}. Therefore, either $\Ru(G)=\{1\}$ or $\Ru(G)\simeq W$. In the latter case, since $\dim W = 2$, $G$ is determined by $V^{\diag}$.

Finally, let us show that $\Ru(G)$ is non-trivial. Suppose the contrary: $G$ is reductive. Then $G^\circ$ is reductive and, by \cite[Proposition 2.23]{HMO}, $V$ is not semisimple as a $G^\circ$-module, because $\cV$ is indecomposable. Since $G^\circ$ is non-commutative, it contains a non-trivial semisimple part $S$. Note that $S$ embeds into $D:=\Gl(W_1)\times\Gl(W_2)$, therefore into $[D,D]\simeq\SL_2$. By  \cite[Theorem 19]{Cassimpl}, either $S\simeq\SL_2$ or $S\simeq\SL_2(C)$.  

If $V$ were semisimple as an $S$-module, it would also be semisimple as a $G^\circ$-module. Indeed, let $W'$ be an $S$-submodule of $V$ such that $V = W_2\oplus W'$ as $S$-modules. Since $G^\circ=Z\cdot S$ and $W'$ and $W_2$ are simple non-isomorphic $S$-modules, $Z$ and, therefore, $G^\circ$ must preserve each of them. Since $V$ is not semisimple as a $G^\circ$-module, we conclude that $V$ is not semisimple as an $S$-module. Since all differential representations of $\SL_2(C)$ are algebraic, they are also completely reducible. On the other hand, by \cite[Theorem~4.11]{MinOvRepSL2}, every indecomposable differential representation of $\SL_2$ of dimension~$3$ is irreducible.
\end{proof}

\subsection{Indecomposable upper-triangular case}\label{sec:idecuppertriang}
It remains to consider the case in which there exists a $G$-invariant filtration
\begin{equation}\label{eq:flag} 0=V_0\subset V_1\subset V_2\subset V_3=V,\end{equation}
 where $\dim V_r=r$. 
 Denote the subgroup of $\Gl(V)$ preserving the flag~\eqref{eq:flag} by $B$. We have $G\subset B$. Let us choose an ordered basis $E:=\{e_1,e_2,e_3\}$ of $V$ such that $V_r$ is spanned by $e_1,\ldots,e_r$. With respect to $E$, $B$ can be identified with the subgroup of upper triangular matrices in $\Gl_3(\k)$. Let $B'$ denote the subgroup of $B$ consisting of the elements preserving $\k e_2$.

We will use the notation for the cases from Proposition~\ref{prop:2solvable} to describe the cases for $\cV$. For example, we say that $\cV$ is of type (CQ,CR) if $\cV_2$ and $\cV/\cV_1$ correspond to CQ and CR, respectively. There are $3\times 3=9$ such pairs. 
\begin{itemize}[leftmargin=0.4cm,itemindent=0cm,labelsep=-0.6cm,itemsep=0.1cm,align=left]
\item Due to the duality $\cV\leftrightarrow\cV^*$, it suffices to consider only $6$ cases (e.g., we do not need to treat (CQ,CR) once we have done (CR,CQ)).  
\item The case (CQ,CQ) implies that $\cV^{\diag}$ is constant, so $\tau(G)=0$. 
\item The case (CR,CR) implies that $\cV$ is decomposable, hence we can use Proposition~\ref{prop:decomposable} to deal with this case. We will now assume that $\cV$ is indecomposable. 
\end{itemize}
So, it remains to consider the cases (CR,CQ),  (CR,NC), (CQ,NC), and (NC,NC). 
\begin{itemize}[leftmargin=0.5cm,itemindent=0cm,labelsep=-0.6cm,itemsep=0.1cm,align=left]
\item
If $\cV_2$ is semisimple, denote the invariant complement to $\cV_1$ in $\cV_2$ by  $\cU$. Then  $\cV/\cU$ is either of type (CQ) or (NC), otherwise $\cV$ would have been decomposable.
\begin{itemize}[leftmargin=0.5cm,itemindent=0cm,labelsep=-0.8cm,itemsep=0.1cm,align=left]
\item Since the module $\omega(\cV/\cV_1\oplus\cV/\cU)$ is faithful, the case (CR,CQ,CQ) -- we include the type of $\cV/\cU$ in the end -- implies $\tau(G)=0$.
\item For the case (CR,CQ,NC), we can identify $G$ with a subgroup of $B'$. Then $G$ is determined by $(B',\omega(\cV_1\oplus\cV/\cV_1))$.  Indeed, it is sufficient to show that $G$ contains the group
\begin{equation}\label{eq:Z}
Z=\left\{\begin{pmatrix}
1 & 0 &a\\
0 & 1 & 0\\
0 & 0& 1
\end{pmatrix},\ a\in\k\right\}.
\end{equation}
which is the kernel of the action of $B'$ on $W := \omega(\cV_1\oplus\cV/\cV_1)$. Since the image of $G$ in $\GL(W)$ has differential type $0$ by the assumption and Proposition~\ref{prop:commutative}, 
and, on the other hand, $\tau(G)=1$ by the NC part of (CR,CQ,NC), the kernel of the action of $G$ on $W$ (which belongs to $Z$) has differential type $1$. Then it coincides with $Z$, since every proper subgroup of $\Ga$ has differential type $0$~\cite[Example 2.7(1)]{CassSingerJordan}.
\item The case (CR,NC,CQ) becomes (CR,CQ,NC), which we have considered, after the permutation of $V_1$ and $U$.
\item  It remains to consider the case (CR,NC,NC). Let $g \in G$. Then there exist polynomial functions $a,c,e : G \to \k^\times$ and differential polynomial functions $b,d : G \to \k$ such that, for all $g\in G$,
\begin{equation}\label{eq:G}
g=\begin{pmatrix}
a(g)& 0 &b(g)\\
0 &c(g)&d(g)\\
0 & 0& e(g)
\end{pmatrix}.
\end{equation}
We claim that $G$ is determined by $\omega(\cV_2\oplus\cV/\cV_2)$. We will show this for a more general situation: the polynomial function $\frac{a}{c}$ can be taking values in $C^\times$ (we will {\em use this later} dealing with the case (CQ,NC)). 
It is sufficient to show that $G$ contains 
\begin{equation}\label{eq:Y}
Y:=\left\{y(u,v) :=\begin{pmatrix}  
1& 0 &u\\
0 &1&v\\
0 & 0& 1
\end{pmatrix},\ u,v\in \k\right\}.
\end{equation}
Note that $\tau(G)=1$ by the NC conditions. It follows that $\tau(\Ru(G))=1$ since $G/\Ru(G)$ is commutative, hence of type 0 by Proposition~\ref{prop:commutative}. In particular, $G$ is not reductive. 
Since the images of  $\frac{a}{e}$ and $\frac{c}{e}$ are not contained in $C^\times$, it follows from~\cite[Proposition~3.21]{HMO} that $\Ru(G)$ is a vector space over $\k$. Since $\Ru(G)$ can be identified with a submodule of $\omega(\cV_2\otimes (\cV/\cV_2)^*)$ by \cite[Lemma~3.6]{HMO}, we have, up to a permutation of $e_1$ and $e_2$, the following possibilities:
\begin{enumerate}[leftmargin=0.6cm,itemindent=0cm,labelsep=-1.1cm,itemsep=0.1cm,align=left]
 \item\label{case1} $\Ru(G)=Y$;
 \item\label{case2} $\Ru(G)=\{y(u,0),\ u\in \k\} \subset Y$;
 \item\label{case3} $a=c$ and $\Ru(G)\cong\Ga$. 
\end{enumerate}
In case~\eqref{case1}, there is nothing left to prove. Case~\eqref{case3} reduces to case~\eqref{case2} by a suitable choice of basis of $V$. Suppose we have case~\eqref{case2}. By the NC condition, $\Ru(\Galdelta(\cV/\cV_1))\cong\Ga$. 
Therefore, for every $h \in \k$, there exist $h_1\in\k^\times$ and $h_2 \in\k$ such that
$$
z := \begin{pmatrix}
h_1 & 0 &h_2\\
0 &1&h\\
0 & 0& 1
\end{pmatrix} \in G.
$$
Then, for all $h \in \k$,
\begin{equation}\label{eq:h1h}
y(-h_2,0)z =\begin{pmatrix}
h_1 & 0 &0\\
0 &1&h\\
0 & 0& 1
\end{pmatrix} \in G.
\end{equation}
Let $$
\Gamma = G\cap\left\{\begin{pmatrix}f_1&0&0\\
0&1&f\\
0&0&1
\end{pmatrix},\ f_1 \in \k^\times, f \in \k\right\},
$$
which is a differential algebraic subgroup with $\tau(\Gamma)=1 $ by~\eqref{eq:h1h}. Since the restriction of $\Gamma$ to $V_1$, being a differential algebraic subgroup of the restriction of $G$ to $V_1$, is of type 0 and by \cite[Remark~2.7.5]{CassSingerJordan}, $$\Gamma_0= \left\{\begin{pmatrix}1&0&0\\
0&1&f\\
0&0&1
\end{pmatrix},\  f \in \k\right\},$$ which is normal in $G$, unipotent, and not contained in $\Ru(G)$, as we are in case~\eqref{case2}. Contradiction.
\end{itemize}
\item It remains to consider (NC,NC) and (CQ,NC).\\
\end{itemize}

\begin{proposition}\label{prop:commutator}
Let $G\subset\Gl_n(\k)$ be a subgroup and $\overline{G}$ its Zariski closure. Then $$[\overline{G},\overline{G}]=\overline{[G,G]}.$$ 
\end{proposition}
\begin{proof}
Since the commutator group of a linear algebraic group is Zariski closed \cite[Proposition 17.2]{Humphline}, we have 
$$\overline{[G,G]}\subset[\overline{G},\overline{G}].$$
The other inclusion follows immediately from~\cite[Theorem 4.3(c)]{Waterhouse:IntrotoAffineGroupSchemes}.
\end{proof}

\begin{proposition}\label{prop:ncnc}
Suppose that $G$ is of type (NC,NC).
\begin{enumerate}[leftmargin=0.5cm,itemindent=0cm,labelsep=0.55cm,itemsep=-0.25cm,align=parleft]
 \item If $[G,G]$ is commutative, then $G$ is determined by $V_2$.\\
 \item If $[G,G]$ is not commutative, then $G$ is determined by $V^{\diag}$.
\end{enumerate}
\end{proposition}
\begin{proof}
Since $G$ is of type (NC,NC), $\tau(G) = 1$.
Since $G$ is solvable, ${(G/\Ru(G))}^\circ$, being a connected solvable reductive LDAG, is a $\delta$-torus.
Hence, by 
Proposition~\ref{prop:commutative}, $\tau(G/\Ru(G))=0$. Therefore, by definition, $G_0 \subset \Ru(G)$.

Let $\ra$ stand for the Lie algebra of $\Ru(G)$ (see \cite[Section~4.1]{HMO} for a quick and sufficient overview) and Z be defined by~\eqref{eq:Z}.
 Then, under the matrix conjugation, $$A:=\ra/(\ra\cap\Lie{Z})$$ is a $G$-invariant $\delta$-subgroup of the semisimple (2-dimensional) $G$-module $$L:=\Lie{[B,B]}/\Lie{Z}$$ (recall that $B$ is the group of all upper triangular matrices, and so $\Lie[B,B]$ is the Lie algebra of upper triangular matrices with $0$ on the diagonal). By the (NC,NC) assumption and since, for all $a,b$, $a_{ij}$, $1 \leqslant i \leqslant j \leqslant 3$,
$$\begin{pmatrix}
a_{11}&a_{12}&a_{13}\\
0&a_{22}&a_{23}\\
0&0&a_{33}
\end{pmatrix}^{-1}
\begin{pmatrix}
0&a&0\\
0&0&b\\
0&0&0
\end{pmatrix}
\begin{pmatrix}
a_{11}&a_{12}&a_{13}\\
0&a_{22}&a_{23}\\
0&0&a_{33}
\end{pmatrix}=\begin{pmatrix}
0&\frac{a_{22}}{a_{11}}a&\frac{a_{23}}{a_{11}}a-\frac{a_{12}a_{33}}{a_{11}a_{22}}b\\
0&0&\frac{a_{33}}{a_{22}}b\\
0&0&0
\end{pmatrix},
$$
$G$ acts on $L$ purely non-constantly (Definition~\ref{def:purconst}). Therefore, by \cite[Proposition~3.21]{HMO},  $A$ is a $\k$-subspace of $L$. 
By \cite[Proposition~22]{Cassidy:differentialalgebraicgroups}, the projections of $A$ onto the $(1,2)$ and $(2,3)$ entries are non-trivial.  Hence,
if $A\neq L$, then the action of $G$ on $L$ is isotypic: 
\begin{equation}\label{eq:same}
\frac{a_{11}}{a_{22}}=\frac{a_{22}}{a_{33}}
\end{equation} 
on $G$. Moreover, if $A\neq L$, there exists $c\in \k^\times$ such that 
\begin{equation}\label{eq:matrix}
\ra \subset\left\{
\begin{pmatrix}
0 & a & b\\
0&0&ca\\
0&0&0
\end{pmatrix},\ a,b \in\k\right\}.
\end{equation} Therefore, $\ra$ is a commutative Lie algebra, and so $\Ru(G)$ is commutative.  
If $A=L$, then $\ra\supset [\ra,\ra]=\Lie Z$. 
Therefore, if $A=L$, then $\ra=\Lie[B,B]$, and so  $\Ru(G)$ is the whole group of unipotent matrices $[B,B]$.
Since $[G,G]$ is normal and unipotent, $$[G,G]\subset\Ru(G).$$ Therefore, if $[G,G]$ is non-commutative, so is $\Ru(G)$. By the above, this implies that $$[B,B]\subset G.$$ Hence, $G$ is determined by $V^{\diag}$.

Suppose now that $[G,G]$ is commutative. We will show that
\begin{equation}\label{eq:radical}
\Ru(G)=\Ru(\overline{G}).
\end{equation}
It would follow then that $G$ is determined by $V_2$. Indeed, let $\varrho:\overline{G}\to\GL(V_2)$ be the restriction to $V_2$ of the natural representation of $\overline{G}$ on $V$.   \eqref{eq:same} implies that $\Ker\varrho$ consists of unipotent matrices. Hence, $\Ker\varrho \subset \Ru(\overline{G})$. If~\eqref{eq:radical} holds, then
$$
G\subset\varrho^{-1}(\varrho(G)) = G\cdot\Ker\varrho \subset G\cdot\Ru(\overline{G})=G\cdot\Ru(\overline{G})=G.
$$
Therefore, $G$ is determined by $V_2$.

As we have shown above, $A\neq L$ and $\frac{a_{11}}{a_{22}}=\frac{a_{22}}{a_{33}}$ on $G$. In particular, the function $$\frac{a_{11}}{a_{33}}=(a_{11}/a_{22})^2$$ from $G$ to $\k^\times$ does not have its image contained in $C^\times$. This is the character of the action of $G$ on $\Lie{Z}$ by conjugation.
Since $\Lie G\cap\Lie Z$ is $G$-invariant,  \cite[Proposition~3.21]{HMO} implies that $\Lie G\cap\Lie{Z}$ is a $\k$-subspace of $\Lie Z$, and so  either $Z\subset G$ or $Z\cap G=\{1\}$. 

Suppose first that $Z\subset G$. Since  $[G,G]$ is commutative, Proposition~\ref{prop:commutator} implies that $[\overline{G},\overline{G}]$ is commutative. This is possible only if $\Lie\Ru(\overline{G})/\Lie{Z}$ is proper in $L$: if $\Lie\Ru(\overline{G})/\Lie{Z}=L$, $\Ru(\overline{G})=[B,B]$, which is not commutative. Since $\Lie\Ru(\overline{G})/\Lie{Z}$ contains $A$, it therefore coincides with $A$ (all non-zero proper $\k$-subspaces of $L$ are maximal since $\dim_\k L=2$). Since $Z\subset\Ru(G)$, we conclude that the Lie algebras of $\Ru(G)$ and of $\Ru(\overline{G})$ coincide. Hence, by~\cite[Proposition 26]{Cassidy:differentialalgebraicgroups}, \eqref{eq:radical} holds.

It remains to consider the case in which $[G,G]$ is commutative and $Z\cap G=\{1\}$. By rescaling $e_3$, we will assume that $c=1$ in~\eqref{eq:matrix}. Hence, there exists a differential polynomial $\varphi(x)$ such that 
\begin{equation}\label{eq:ru}
\Ru(G)=\left\{\begin{pmatrix}
1 & x & \varphi(x)\\
0 & 1 & x\\
0 & 0 & 1
\end{pmatrix},\ x\in\k\right\}.
\end{equation}
Since $\Ru(G)$ is a group, for all $x,y \in \k$,
\begin{equation}\label{eq:phixy}
\varphi(x+y)=\varphi(x)+\varphi(y)+xy.
\end{equation}
For all $x\in\k$, define $\widetilde\varphi(x)=\varphi(x)-\frac{x^2}{2}$. Then~\eqref{eq:phixy} is equivalent to, for all $x,y\in\k$, $\widetilde\varphi(x+y)=\widetilde\varphi(x)+\widetilde\varphi(y)$. Therefore, $\widetilde\varphi$ is a homogeneous linear differential polynomial.
Let $\gamma_i \in \k$ be such that $\widetilde\varphi=\sum\limits_i\gamma_ix^{(i)}$. By the basis change $e_2\mapsto e_2-\gamma_0e_1$, we may assume that $\gamma_0=0$. 
For all $a,\lambda\in\k^\times$ and $b,c,d \in \k$, let
$$
M(\lambda,a,b,c,d)=\begin{pmatrix}
a\lambda^2 &a b & ac\\
0 & a\lambda & ad\\
0 & 0 & a
\end{pmatrix}.
$$
Let $N$ be  the normalizer  of $\Ru(G)$ in $B$. A computation shows that
$$
N = \left\{M(\lambda,a,b,c,d)\:|\:\lambda,a\in\k^\times, b,c,d\in\k:\forall x\in\k\ \widetilde\varphi(\lambda x)=\lambda^2\widetilde\varphi(x)+(b-\lambda d)x\right\}.
$$
Let $f : G\to \k^\times$ be the homomorphism given by $\frac{a_{22}}{a_{33}}$. By the NC condition, $f(G)$ is a Zariski dense $\delta$-subgroup of $\k^\times$. By \cite[Proposition~31]{Cassidy:differentialalgebraicgroups}, $C^\times \subset f(G)$.  Since $G \subset N$, for all $\lambda \in C^\times$, there exist $a\in\k^\times$, $b,c,d\in\k$ such that $M(\lambda,a,b,c,d) \in G$, which implies that $\widetilde\varphi =0$. Therefore, 
\begin{equation}\label{eq:Ruoverline}
\Ru(G)=\overline{\Ru(G)}
\end{equation}  and, for all $\lambda,a \in \k^\times$ and $b,c,d \in \k$, $M(\lambda,a,b,c,d) \in N$ implies that $b=\lambda d$. Hence, for all $a,b,c,d,e \in \k$, if
$$
\begin{pmatrix}
a& 0 &b\\
0& c&d\\
0&0&e
\end{pmatrix} \in N,
$$
then $d =0$. Let $F \subset G$ be the $\delta$-subgroup defined by $a_{12}=a_{23}=0$.  Since $Z\cap G = \{1\}$, $F\cap\Ru(G) = \{1\}$ and $F$ acts faithfully on $V^{\diag}$, and so $F$ is commutative. Hence, $\overline{F}$ is commutative.
For any $b\in B$, there exists $u \in \Ru(G)$ such that $a_{12}(bu)=0$. Thus,  $G = F\ltimes \Ru(G)$. Since there exist $a,b,c,d \in \k$ such that
$M(2,a,b,c,d) \in G$, there exists $v \in \Ru(G)$ and $a',c'\in\k$ such that $$vM(2,a,b,c,d)=M(2,a',0,c',0) =: m \in F.$$
Suppose that $\Ru(\overline{F})=Z$. Let $1 \ne z \in Z$. Then $m^{-1}zm\ne z$. This contradicts the commutativity of $\overline{F}$.
Therefore, $\overline{F}$ is reductive. By \cite[Corollary~7.4]{Humphline} and~\eqref{eq:Ruoverline}, we conclude that $$\overline{G} = \overline{F}\ltimes \overline{\Ru(G)}=\overline{F}\ltimes \Ru(G).$$
Thus,~\eqref{eq:radical} holds. 
\end{proof}

\begin{lemma}\label{lem:firstorderphipsi}
Let $A\subset \k^\times$ be a Kolchin closed subgroup that has a non-constant element. If $\varphi: A\to \k$ and $\psi: \k\to \k$ are differential polynomial maps such that
\begin{align}
\varphi(ab) &=a\varphi(b)+b\varphi(a)\qquad\forall\ a,b\in A\label{eq:58}\\
\psi(ax) &=a\psi(x)+x\varphi(a)\qquad\forall\ a\in A,\ x\in k\label{eq:59},
\end{align} 
then there exist $a_0, a_1 \in \k$ such that, for all $x \in \k$ and $a \in A$, $$\psi(x)=a_0x+a_1x'\quad \text{and}\quad \varphi(a)=a_1a'.$$
\end{lemma}
\begin{proof}
Note that $A$ has to be Zariski dense in $k$, and therefore has to contain all non-zero elements of the constants $C\subset k$ \cite[Proposition~31]{Cassidy:differentialalgebraicgroups}. 
Condition~\eqref{eq:58} implies that the restriction of $\varphi$ to $\Q^\times$ is a derivation from $\Q^\times$ to $k$ and, therefore, is the zero map (using a standard argument). This implies that $\varphi(C^\times)=\{0\}$, because the restriction of $\varphi$ to $C^\times$ is polynomial. 
Therefore, condition~\eqref{eq:59} implies that, for all $c\in C^\times$ and $x \in k$, \begin{equation}\label{eq:psicx} \psi(cx)=c\psi(x).\end{equation} 
Hence, $\psi(0)=\psi(2\cdot 0) = 2\cdot \psi(0)$, and so $\psi(0)=0$.
Let $h$ be the order of $\psi$. In coordinates, $\psi(y) \in k[y,y',\ldots,y^{(h)}]$. Then~\eqref{eq:psicx} implies that $\deg \psi\leq 1$. Let  $a_0,\ldots,a_h \in k$ be such that $$\psi(y) = \sum\limits_{i=0}^h a_iy^{(i)}.$$
Condition~\eqref{eq:59} also implies that, for all $a \in A$, 
\begin{equation}\label{eq:phipsi}
\varphi(a)=\psi(a)-a\psi(1).
\end{equation} Hence, for all $a,\:b\in A$, 
$$
\psi(ab)=a\psi(b)+b\psi(a)-ab\psi(1).
$$
We then have, for all $a,\:b\in A$, 
$$
\sum\limits_{i=0}^h a_i\sum_{j=0}^i\binom{i}{j}a^{(j)}b^{(i-j)}= \sum\limits_{i=0}^h a_i\left(ab^{(i)}+ba^{(i)}\right)-aba_0,
$$
which implies that
$$
\sum\limits_{i=2}^h a_i\sum_{j=0}^i\binom{i}{j}a^{(j)}b^{(i-j)}= \sum\limits_{i=2}^h a_i\left(ab^{(i)}+ba^{(i)}\right).
$$
Hence,  for all $a,\:b\in A$, 
$$
\sum\limits_{i=2}^h a_i\sum_{j=1}^{i-1}\binom{i}{j}a^{(j)}b^{(i-j)}= 0.
$$
For each non-constant $a \in A$, this implies that the non-zero linear differential polynomial
$$f(y) := \sum\limits_{i=2}^h a_i\sum_{j=1}^{i-1}\binom{i}{j}a^{(j)}y^{(i-j)}
$$
is in the defining ideal of $A$. Let $b \in A$ be another non-constant element. Since $C$ is algebraically closed, $b$ is transcendental over $C$, and so $b,b^2,\ldots, b^h$ are linearly independent over $C$ and $f(b^i)=0$ for all $i$. This contradicts with $f(y)=0$, being an $(h-1)$-st order linear differential equation over $k$, having a solution space that is $(h-1)$-dimensional over $C$. Thus, $h=1$, and so $\psi(y)=a_0y+a_1y'$. Finally, by~\eqref{eq:phipsi}, for all $a\in A$, $\varphi(a) = a_0a+a_1a'-aa_0=a_1a'$.
\end{proof}

\begin{proposition}\label{prop:cqnc}
Suppose that $G$ is of type (CQ,NC). Then, either
\begin{enumerate}[leftmargin=0.5cm,itemindent=0cm,labelsep=0.55cm,itemsep=0.1cm,align=parleft]
\item $G$ is determined by $\omega(\cV_2\oplus\cV/\cV_2)$ or
\item  $\cV_2$ is not semisimple, the restriction of $G$ to $\omega\left(\cV_2\otimes\left(\cV/\cV_2\right)^*\right)$ is reductive,  and $\cV$ belongs to the tensor category generated by the first prolongation of $\cV/\cV_1$.
\end{enumerate}
\end{proposition}
\begin{proof}
Since the restriction of $G$ to $\omega(\cV_2\oplus\cV/\cV_2)$ is of type $0$,  the NC condition implies that $G_0$ is contained in $Y$ defined by~\eqref{eq:Y}. Moreover, the restriction of $G_0$ to $\omega(\cV/\cV_1)$ is isomorphic to $\Ga$. 

The character of the action of $G$ on $\Lie{Z}$, where $Z$ is defined by~\eqref{eq:Z}, by conjugation equals $\frac{a_{11}}{a_{33}}$, whose image is not contained in $C^\times$ by the (CQ,NC) assumption.
Since $\Lie G\cap\Lie Z$ is $G$-invariant,  \cite[Proposition~3.21]{HMO} implies that $\Lie G\cap\Lie{Z}$ is a $\k$-subspace of $\Lie Z$, and so  either $Z\subset G$ or $Z\cap G=\{1\}$.

If $\cV_2$ is semisimple, then $G$ is determined by $\omega(\cV_2\oplus\cV/\cV_2)$, as was noticed while dealing with the case (CR,NC,NC). 
From now on, we will assume that $\cV_2$ is not semisimple. 
Let $G'$ denote the restriction of $G$ to $\omega\left(\cV_2\otimes\left(\cV/\cV_2\right)^*\right)$. Suppose that $\Ru(G')\neq\{1\}$. Hence, there exist $x,y,z,t,u,v \in \k$ such that $xt\ne 0$, $v \ne 0$,
$$
g_1 := \begin{pmatrix}
t & xt & yt\\
0 & t & zt\\
0 & 0 & t
\end{pmatrix} \in G,\qquad\text{and}\qquad
g_2 := \begin{pmatrix}
1 & 0 & u\\
0 & 1 & v\\
0 & 0 & 1
\end{pmatrix}\in G.
$$ 
Since $xv \ne 0$ and
$$
1\ne
g_1g_2g_1^{-1}g_2^{-1}=\begin{pmatrix}
1 & 0 & xv\\
0 & 1 & 0\\
0 & 0 & 1
\end{pmatrix}\in Z,
$$
 $Z\subset G$. Hence, $G$ contains $Y=ZG_0$ and, thus, is determined by $\omega(\cV_2\oplus\cV/\cV_2)$.

It remains to consider the case of reductive $G'$. Since $[G',G']$ is unipotent and, therefore, connected, $[G',G']\subset\Ru(G')=\{1\}$. Hence, $\overline{G'}$ is commutative. Since $\cV_2$ is not semisimple, we conclude that $\cV_1\cong\cV_2/\cV_1$. Thus, there exist a Kolchin closed subgroup $A \subset \k^\times$,  $A \ne C^\times$, and  a differential polynomial $\varphi$ of positive order (it cannot be replaced by a usual polynomial)  such that, for every $g \in G$, there exist $t,u,v \in \k$ and $a \in A$ such that $t \ne 0$,
\begin{equation}\label{eq:gr}
g =\begin{pmatrix}
at & \varphi(a)t & ut\\
0 & at & vt\\
0 & 0 & t
\end{pmatrix},
\end{equation}
and, for all $a \in A$, there exist $t,u,v \in \k$, such that
$$
\begin{pmatrix}
at & \varphi(a)t & ut\\
0 & at & vt\\
0 & 0 & t
\end{pmatrix} \in G
$$ 
 If $Z \subset G$, then $G$ contains $Y=ZG_0$ and, therefore, is determined by $\omega(\cV_2\oplus\cV/\cV_2)$. If $Z\not\subset G$, then there exists a differential polynomial $\psi$ such that
$$
\Ru(G)=\left\{\begin{pmatrix}
1 & 0 & \psi(x)\\
0 & 1 & x\\
0 & 0 & 1
\end{pmatrix}\Big|\: x \in \k\right\}.
$$ Since the conjugation by~\eqref{eq:gr} preserves $\Ru(G)$, we obtain the restriction on $\psi$:
$$
\psi(ax)=a\psi(x)+x\varphi(a)\qquad\forall\ a\in A,\ x\in \k.
$$
Hence, by Lemma~\ref{lem:firstorderphipsi}, there exist $a_0 \in \k$ and $a_1 \in \k^\times$ such that $\psi(x)=a_0x+a_1x'$ and $\varphi(a)=a_1a'$ for all $x \in \k$ and $a \in A$. 
By the change of basis with the matrix
$$
\begin{pmatrix}
1 & \frac{a_0}{a_1} &0\\
0 & \frac{1}{a_1} &0\\
0 & 0 & \frac{1}{a_1}
\end{pmatrix}
$$
in representation~\eqref{eq:gr}, $\varphi(a)$ and $u$ in~\eqref{eq:gr} can be replaced by $a'$ and $v'$, respectively. Hence, $\cV$ is in the tensor category generated by $(\cV/\cV_1)^{(1)}$.
\end{proof}
\section{Acknowledgements}
The authors are grateful to Charlotte Hardouin for productive discussions and useful suggestions.

\bibliographystyle{abbrvnat}
\bibliography{bibdata}

\def\cprime{$'$}
\begin{thebibliography}{33}
\providecommand{\natexlab}[1]{#1}
\providecommand{\url}[1]{\texttt{#1}}
\expandafter\ifx\csname urlstyle\endcsname\relax
  \providecommand{\doi}[1]{doi: #1}\else
  \providecommand{\doi}{doi: \begingroup \urlstyle{rm}\Url}\fi

\bibitem[Arreche(2014)]{CarlosISSAC}
C.~Arreche.
\newblock Computing the differential {G}alois group of a parameterized
  second-order linear differential equation.
\newblock In \emph{Proceedings of the 39th International Symposium on Symbolic
  and Algebraic Computation, ISSAC 2014}, pages 43--50, New York, 2014. ACM
  Press.
\newblock URL \url{http://dx.doi.org/10.1145/2608628.2608680}.

\bibitem[Arreche(2016)]{ArrecheJSC}
C.~Arreche.
\newblock On the computation of the parameterized differential {G}alois group
  for a second-order linear differential equation with differential parameters.
\newblock \emph{Journal of Symbolic Computation}, 75:\penalty0 25--55, 2016.
\newblock URL \url{http://dx.doi.org/10.1016/j.jsc.2015.11.006}.

\bibitem[Cassidy(1972{\natexlab{a}})]{Cassidy:differentialalgebraicgroups}
P.~Cassidy.
\newblock Differential algebraic groups.
\newblock \emph{American Journal of Mathematics}, 94:\penalty0 891--954,
  1972{\natexlab{a}}.
\newblock URL \url{http://www.jstor.org/stable/2373764}.

\bibitem[Cassidy(1972{\natexlab{b}})]{cassdiffgr}
P.~Cassidy.
\newblock Differential algebraic groups.
\newblock \emph{American Journal of Mathematics}, 94:\penalty0 891--954,
  1972{\natexlab{b}}.
\newblock URL \url{http://www.jstor.org/stable/2373764}.

\bibitem[Cassidy(1977)]{Cassunip}
P.~Cassidy.
\newblock Unipotent differential algebraic groups.
\newblock In \emph{Contributions to algebra: Collection of papers dedicated to
  {E}llis {K}olchin}, pages 83--115. Academic Press, 1977.

\bibitem[Cassidy(1989)]{Cassimpl}
P.~Cassidy.
\newblock The classification of the semisimple differential algebraic groups
  and linear semisimple differential algebraic {L}ie algebras.
\newblock \emph{Journal of Algebra}, 121\penalty0 (1):\penalty0 169--238, 1989.
\newblock URL \url{http://dx.doi.org/10.1016/0021-8693(89)90092-6}.

\bibitem[Cassidy and Singer(2011)]{CassSingerJordan}
P.~Cassidy and M.~Singer.
\newblock A {J}ordan--{H}{\"o}lder theorem for differential algebraic groups.
\newblock \emph{Journal of Algebra}, 328\penalty0 (1):\penalty0 190--217, 2011.
\newblock URL \url{http://dx.doi.org/10.1016/j.jalgebra.2010.08.019}.

\bibitem[Cassidy and Singer(2007)]{cassisinger}
P.~Cassidy and M.~F. Singer.
\newblock Galois theory of parametrized differential equations and linear
  differential algebraic group.
\newblock \emph{IRMA Lectures in Mathematics and Theoretical Physics},
  9:\penalty0 113--157, 2007.
\newblock URL \url{http://dx.doi.org/10.4171/020-1/7}.

\bibitem[Deligne(1990)]{Deligne:categoriestannakien}
P.~Deligne.
\newblock Cat\'egories tannakiennes.
\newblock In \emph{The Grothendieck Festschrift, Volume~II}, Modern
  Birkh\"auser Classics, pages 111--195. Birkh\"auser, Boston, MA, 1990.
\newblock URL \url{http://dx.doi.org/10.1007/978-0-8176-4575-5}.

\bibitem[Dreyfus(2014{\natexlab{a}})]{Dreyfus}
T.~Dreyfus.
\newblock Computing the {G}alois group of some parameterized linear
  differential equation of order two.
\newblock \emph{Proceedings of the American Mathematical Society},
  142:\penalty0 1193--1207, 2014{\natexlab{a}}.
\newblock URL \url{http://dx.doi.org/10.1090/S0002-9939-2014-11826-0}.

\bibitem[Dreyfus(2014{\natexlab{b}})]{Dreyfus:density}
T.~Dreyfus.
\newblock A density theorem in parameterized differential {G}alois theory.
\newblock \emph{Pacific Journal of Mathematics}, 271\penalty0 (1):\penalty0
  87--141, 2014{\natexlab{b}}.
\newblock URL \url{http://dx.doi.org/10.2140/pjm.2014.271.87}.

\bibitem[Feng(2015)]{RFdifferential2015}
R.~Feng.
\newblock Hrushovski's algorithm for computing the {G}alois group of a linear
  differential equation.
\newblock \emph{Advances in Applied Mathematics}, 65:\penalty0 1--37, 2015.
\newblock URL \url{http://dx.doi.org/10.1016/j.aam.2015.01.001}.

\bibitem[Gillet et~al.(2013)Gillet, Gorchinskiy, and Ovchinnikov]{GGO}
H.~Gillet, S.~Gorchinskiy, and A.~Ovchinnikov.
\newblock Parameterized {P}icard--{V}essiot extensions and {A}tiyah extensions.
\newblock \emph{Advances in Mathematics}, 238:\penalty0 322--411, 2013.
\newblock URL \url{http://dx.doi.org/10.1016/j.aim.2013.02.006}.

\bibitem[Gorchinskiy and Ovchinnikov(2014)]{GO}
S.~Gorchinskiy and A.~Ovchinnikov.
\newblock Isomonodromic differential equations and differential categories.
\newblock \emph{Journal de Math\'ematiques Pures et Appliqu\'ees},
  102:\penalty0 48--78, 2014.
\newblock URL \url{http://dx.doi.org/10.1016/j.matpur.2013.11.001}.

\bibitem[Hardouin and Singer(2008)]{HardouinSinger}
C.~Hardouin and M.~F. Singer.
\newblock Differential {G}alois theory of linear difference equations.
\newblock \emph{Mathematische Annalen}, 342\penalty0 (2):\penalty0 333--377,
  2008.
\newblock URL \url{http://dx.doi.org/10.1007/s00208-008-0238-z}.

\bibitem[Hardouin et~al.(2016)Hardouin, Minchenko, and Ovchinnikov]{HMO}
C.~Hardouin, A.~Minchenko, and A.~Ovchinnikov.
\newblock Calculating differential {G}alois groups of parametrized differential
  equations, with applications to hypertranscendence.
\newblock \emph{Mathematische Annalen}, 2016.
\newblock URL \url{http://dx.doi.org/10.1007/s00208-016-1442-x}.

\bibitem[Hrushovski(2002)]{Hrushcomp}
E.~Hrushovski.
\newblock Computing the {G}alois group of a linear differential equation.
\newblock In \emph{Differential {G}alois theory ({B}edlewo, 2001)}, volume~58
  of \emph{Banach Center Publ.}, pages 97--138. Polish Acad. Sci., Warsaw,
  2002.
\newblock URL \url{http://dx.doi.org/10.4064/bc58-0-9}.

\bibitem[Humphreys(1975)]{Humphline}
J.~E. Humphreys.
\newblock \emph{Linear algebraic groups}.
\newblock Springer-Verlag, New York, 1975.
\newblock URL \url{http://dx.doi.org/10.1007/978-1-4684-9443-3}.

\bibitem[Kamensky(2015)]{MosheTAMS}
M.~Kamensky.
\newblock Model theory and the {T}annakian formalism.
\newblock \emph{Transactions of the American Mathematical Society},
  367:\penalty0 1095--1120, 2015.
\newblock URL \url{http://dx.doi.org/10.1090/S0002-9947-2014-06062-5}.

\bibitem[Kaplansky(1957)]{Kapldiffalg}
I.~Kaplansky.
\newblock \emph{An introduction to differential algebra}.
\newblock Hermann, Paris, 1957.

\bibitem[Landau and Lifshitz(1980)]{LL}
L.~Landau and E.~Lifshitz.
\newblock \emph{The Classical Theory of Fields}, volume~2 of \emph{Course of
  Theoretical Physics}.
\newblock Butterworth--Heinemann, 4rd revised {E}nglish edition, 1980.

\bibitem[Marker(2000)]{Marker2000}
D.~Marker.
\newblock Model theory of differential fields.
\newblock In \emph{Model theory, algebra, and geometry}, volume~39 of
  \emph{Mathematical Sciences Research Institute Publications}, pages 53--63.
  Cambridge University Press, Cambridge, 2000.
\newblock URL \url{http://library.msri.org/books/Book39/files/dcf.pdf}.

\bibitem[Minchenko and Ovchinnikov(2011)]{diffreductive}
A.~Minchenko and A.~Ovchinnikov.
\newblock Zariski closures of reductive linear differential algebraic groups.
\newblock \emph{Advances in Mathematics}, 227\penalty0 (3):\penalty0
  1195--1224, 2011.
\newblock URL \url{http://dx.doi.org/10.1016/j.aim.2011.03.002}.

\bibitem[Minchenko and Ovchinnikov(2013)]{MinOvRepSL2}
A.~Minchenko and A.~Ovchinnikov.
\newblock Extensions of differential representations of {$\SL_2$} and tori.
\newblock \emph{Journal of the Institute of Mathematics of Jussieu},
  12\penalty0 (1):\penalty0 199--224, 2013.
\newblock URL \url{http://dx.doi.org/10.1017/S1474748012000692}.

\bibitem[Minchenko et~al.(2014)Minchenko, Ovchinnikov, and
  Singer]{MinOvSingunip}
A.~Minchenko, A.~Ovchinnikov, and M.~F. Singer.
\newblock Unipotent differential algebraic groups as parameterized differential
  {G}alois groups.
\newblock \emph{Journal of the Institute of Mathematics of Jussieu},
  13\penalty0 (4):\penalty0 671--700, 2014.
\newblock URL \url{http://dx.doi.org/10.1017/S1474748013000200}.

\bibitem[Minchenko et~al.(2015)Minchenko, Ovchinnikov, and Singer]{MinOvSing}
A.~Minchenko, A.~Ovchinnikov, and M.~F. Singer.
\newblock Reductive linear differential algebraic group and the {G}alois groups
  of parametrized linear differential equations.
\newblock \emph{International Mathematics Research Notices}, 2015\penalty0
  (7):\penalty0 1733--1793, 2015.
\newblock URL \url{http://dx.doi.org/10.1093/imrn/rnt344}.

\bibitem[Mitschi and Singer(2012)]{MitSingMonod}
C.~Mitschi and M.~F. Singer.
\newblock Monodromy groups of parameterized linear differential equations with
  regular singularities.
\newblock \emph{Bull. Lond. Math. Soc.}, 44\penalty0 (5):\penalty0 913--930,
  2012.

\bibitem[Nagloo and Le\'on~S\'anchez(2016)]{LSN}
J.~Nagloo and O.~Le\'on~S\'anchez.
\newblock On parameterized differential {G}alois extensions.
\newblock \emph{Journal of Pure and Applied Algebra}, 220\penalty0
  (7):\penalty0 2549--2563, 2016.
\newblock URL \url{http://dx.doi.org/10.1016/j.jpaa.2015.12.001}.

\bibitem[Ovchinnikov(2009)]{OvchTannakian}
A.~Ovchinnikov.
\newblock Tannakian categories, linear differential algebraic groups, and
  parametrized linear differential equations.
\newblock \emph{Transformation Groups}, 14\penalty0 (1):\penalty0 195--223,
  2009.
\newblock URL \url{http://dx.doi.org/10.1007/s00031-008-9042-9}.

\bibitem[Ritt(1950)]{RittDiffalg}
J.~F. Ritt.
\newblock \emph{Differential {A}lgebra}.
\newblock American Mathematical Society Colloquium Publications, Vol. XXXIII.
  American Mathematical Society, New York, N. Y., 1950.

\bibitem[van~der Put and Singer(2003)]{vdPutSingerDifferential}
M.~van~der Put and M.~F. Singer.
\newblock \emph{Galois theory of linear differential equations}.
\newblock Springer-Verlag, Berlin, 2003.
\newblock URL \url{http://dx.doi.org/10.1007/978-3-642-55750-7}.

\bibitem[Waterhouse(1979)]{Waterhouse:IntrotoAffineGroupSchemes}
W.~C. Waterhouse.
\newblock \emph{Introduction to affine group schemes}, volume~66 of
  \emph{Graduate Texts in Mathematics}.
\newblock Springer-Verlag, New York, 1979.
\newblock URL \url{http://dx.doi.org/10.1007/978-1-4612-6217-6}.

\bibitem[Wibmer(2012)]{wibdesc}
M.~Wibmer.
\newblock Existence of {$\partial$}-parameterized {P}icard-{V}essiot extensions
  over fields with algebraically closed constants.
\newblock \emph{Journal of Algebra}, 361:\penalty0 163--171, 2012.
\newblock URL \url{http://dx.doi.org/10.1016/j.jalgebra.2012.03.035}.

\end{thebibliography}
\end{document}